\documentclass[12pt]{amsart}
\usepackage{amsaddr}
\usepackage{amsmath,amsthm,amssymb}
\usepackage{graphicx}
\usepackage{subfig}

\DeclareMathOperator{\RE}{Re}

\numberwithin{equation}{section}
\newtheorem{theorem}{Theorem}[section]
\newtheorem{lemma}[theorem]{Lemma}
\newtheorem{corollary}[theorem]{Corollary}
\theoremstyle{remark}

\begin{document}

\title{A subclass of close-to-convex harmonic mappings }

\thanks{The research work of the first author is supported by research fellowship from Council of Scientific and
Industrial Research (CSIR), New Delhi.}

\author[S. Nagpal]{\bf Sumit Nagpal}
\address{\rm Department of Mathematics\\
 University of Delhi,  Delhi--110 007, India}
\email{sumitnagpal.du@gmail.com}

\author[V. Ravichandran]{\bf V. Ravichandran}

\address{\rm Department of Mathematics\\
 University of Delhi, Delhi--110 007, India; \\
School of Mathematical Sciences\\
Universiti Sains Malaysia, 11800 USM, Penang, Malaysia}
\email{vravi68@gmail.com}

\begin{abstract}
A subclass of complex-valued close-to-convex harmonic functions that are univalent and sense-preserving in the open unit disc is investigated. The coefficient estimates, growth results, area theorem, boundary behavior, convolution and convex combination properties for the above family of harmonic functions are obtained. \end{abstract}


 \maketitle

\section{Introduction}
Let $\mathcal{H}$ denote the class of all complex-valued harmonic functions $f$ in the unit disk $\mathbb{D}=\{z \in \mathbb{C}:|z|<1\}$ normalized by $f(0)=0=f_{z}(0)-1$. Let $\mathcal{S}_{H}$ be the subclass of $\mathcal{H}$ consisting of univalent and sense-preserving functions. Such functions can be written in the form $f=h+\bar{g}$, where
\begin{equation}\label{eq1.1}
h(z)=z+\sum_{n=2}^{\infty}a_{n}z^{n}\quad\mbox{and}\quad g(z)=\sum_{n=1}^{\infty}b_{n}z^n
\end{equation}
are analytic in $\mathbb{D}$ and the Jacobian $J_{f}(z)=|h'(z)|^2-|g'(z)|^2$ is positive or equivalently $|g'(z)|<|h'(z)| $ in $\mathbb{D}$.
The classical family $\mathcal{S}$ of normalized analytic univalent functions is a subclass of $\mathcal{S}_{H}$. The family of all functions $f \in \mathcal{S}_{H}$ with the additional property that $f_{\bar{z}}(0)=0$ is denoted by $\mathcal{S}_{H}^{0}$. Both $\mathcal{S}_{H}$ and $\mathcal{S}_{H}^{0}$ are normal families. But $\mathcal{S}_{H}^{0}$ is the only compact family with respect to the topology of locally uniform convergence.

In $1984$, Clunie and Sheil-Small (see \cite{cluniesheilsmall}) investigated the class $\mathcal{S}_{H}$ as well as its geometric subclasses and its properties. Since then, there have been several studies related to the class $\mathcal{S}_{H}$ and their subclasses. In particular, Jahangiri \cite{jahangiriconvex,jahangiristarlike} discussed the subclasses of $\mathcal{S}_{H}$ consisting of functions which are starlike of order $\alpha$ and convex of order $\alpha$, for $0\leq \alpha <1$. Recall that a function $f\in \mathcal{H}$ is said to be starlike of order $\alpha$, $0\leq \alpha <1$ if for any $r\in (0,1)$, we have
\[\frac{\partial}{\partial \theta}\arg f(r e^{i \theta})> \alpha,\quad 0\leq\theta\leq 2\pi.\]
Similarly, $f\in \mathcal{H}$ is said to be convex of order $\alpha$, $0\leq \alpha <1$ if
\[\frac{\partial}{\partial \theta}\left(\arg \left\{\frac{\partial}{\partial \theta}f(r e^{i \theta})\right\}\right)> \alpha,\quad 0\leq\theta\leq 2\pi,\quad 0<r<1.\]

Let $\mathcal{S}_{H}^{*}$, $\mathcal{K}_{H}$ and $\mathcal{C}_{H}$ be the subclasses of $\mathcal{S}_{H}$ mapping $\mathbb{D}$ onto starlike, convex and close-to-convex domains, respectively, just as $\mathcal{S}^{*}$, $\mathcal{K}$ and $\mathcal{C}$ are the subclasses of $\mathcal{S}$ mapping $\mathbb{D}$ onto their respective domains. Denote by $\mathcal{S}_{H}^{*0}$, $\mathcal{K}_{H}^{0}$ and $\mathcal{C}_{H}^{0}$, the classes consisting of those functions $f$ in $\mathcal{S}_{H}^{*}$, $\mathcal{K}_{H}$ and $\mathcal{C}_{H}$ respectively, for which $f_{\bar{z}}(0)=0$. The coefficient bounds and properties of these subclasses are being discussed in \cite{cluniesheilsmall,duren,sheilsmall,wang}.

In \cite{macgregor}, MacGregor studied a subclass of close-to-convex analytic functions. Denote by $\mathcal{F}$ the class consisting of analytic functions which satisfy $|f'(z)-1|<1$ for $z \in \mathbb{D}$ and are normalized by $f(0)=0=f'(0)-1$. In \cite{ponnusamy}, Kalaj, Ponnusamy and Vuorinen introduced a similar subclass of close-to-convex harmonic mappings. Let $\mathcal{F}_{H}$ be the subclass of $\mathcal{H}$ defined by the condition
\[|f_{z}(z)-1|<1-|f_{\bar{z}}(z)|,\quad z\in \mathbb{D}.\]
Observe that $\mathcal{F}_{H}\subset \mathcal{C}_{H}$ and $\mathcal{F}_{H}$ reduces to $\mathcal{F}$ if the co-analytic part of $f$ is zero. Regarding the functions in class $\mathcal{F}_{H}$, following results have been proved in \cite{ponnusamy}.

\begin{lemma}\label{lem1.1}
Let $f \in \mathcal{F}_{H}$ where $h$ and $g$ are given by \eqref{eq1.1}. Then
\begin{itemize}
  \item [(a)] $f$ is close-to-convex in $\mathbb{D}$.
  \item [(b)] $||a_n|-|b_n||\leq 1/n$ for $n \geq 2$ whenever $b_1=0$.
  \item [(c)] $\sum_{n=2}^{\infty}n^{2}(|a_n|^{2}+|b_n|^{2})\leq 1-|b_{1}|^{2}$.
\end{itemize}
\end{lemma}

\begin{lemma}\label{lem1.2}
Let $f=h+\bar{g}\in \mathcal{H}$ where $h$ and $g$ are given by \eqref{eq1.1}, and satisfy the condition
\[\sum_{n=2}^{\infty}n(|a_n|+|b_n|)\leq 1-|b_{1}|.\]
Then $f \in \mathcal{F}_{H}$.
\end{lemma}
Lemma \ref{lem1.2} gives a sufficient condition for a function $f \in \mathcal{H}$ to be in $\mathcal{F}_{H}$. Under the hypothesis of Lemma \ref{lem1.2}, Jahangiri (see \cite{jahangiristarlike}) has proved that $f$ is starlike (of order $0$) in $\mathbb{D}$. The subfamily $\mathcal{F}_{H}\subset \mathcal{S}_{H}$ is not affine and linear invariant, that is, if $f=h+\bar{g} \in \mathcal{F}_{H}$ then
\[\frac{\epsilon \overline{f(z)}+f(z)}{1+\epsilon \overline{g'(0)}}\, (|\epsilon|<1)\quad \mbox{and} \quad \frac{f((z+z_0)/(1+\overline{z_0}z))-f(z_0)}{(1-|z_0|^{2})h'(z_0)}\, (|z_0|<1)\]
need not belong to the class $\mathcal{F}_{H}$. Moreover, the class $\mathcal{F}_{H}$ is not preserved under passage to locally uniform limits. To see this, consider the sequence of affine mappings $f_{n}\in \mathcal{F}_{H}$ defined by
\[f_{n}(z)=z+\frac{n}{n+1}\bar{z}.\]
Then $f_{n}(z)\rightarrow 2 \RE z$ in $\mathbb{D}$ locally uniformly in $\mathbb{D}$. This shows, in particular, that $\mathcal{F}_{H}$ is not a compact family. However, the subclass of $\mathcal{F}_{H}$ consisting of functions $f$ for which $f_{\bar{z}}(0)=0$ is a compact normal family. Denote this class by $\mathcal{F}_{H}^{0}$. The property of $\mathcal{F}_{H}^{0}$ being a compact normal family makes it a proper generalization of the family $\mathcal{F}$ introduced by MacGregor \cite{macgregor}. In this paper, we will investigate the properties of functions in the class $\mathcal{F}_{H}^{0}$.

The growth and area theorems for the class $\mathcal{F}_{H}^{0}$ are being discussed in Section \ref{sec2} of this paper. Section \ref{sec3} investigates the boundary behavior of functions in $\mathcal{F}_{H}^{0}$. The radius of convexity and starlikeness of the family $\mathcal{F}_{H}^{0}$ is also determined.

For analytic functions
\[f(z)=z+\sum_{n=2}^{\infty}a_{n}z^{n} \quad (z \in \mathbb{D})\quad\mbox{and}\quad F(z)=z+\sum_{n=2}^{\infty}A_{n}z^{n}\quad (z \in \mathbb{D}),\]
their convolution (or Hadamard product) is defined as
\[(f*F)(z)=z+\sum_{n=2}^{\infty}a_{n}A_{n}z^{n},\quad z\in \mathbb{D}.\]
In the harmonic case, with $f=h+\bar{g}$ and $F=H+\bar{G}$, their harmonic convolution is defined as
\[f*F=h*H+\overline{g*G}.\]
There have been many results about harmonic convolutions (see \cite{cluniesheilsmall,dorff1,dorff2,goodloe}). In section \ref{sec4}, the convolution and convex combination properties of the members of $\mathcal{F}_{H}^{0}$ are examined.

\section{Growth and Area Theorems}\label{sec2}
The first result connects the relation between the classes $\mathcal{F}$ and $\mathcal{F}_{H}^{0}$.
\begin{theorem}\label{relation}
A harmonic function $f=h+\bar{g} \in \mathcal{F}_{H}^{0}$ if and only if the analytic functions $F_{\epsilon}=h+\epsilon g$ belongs to $\mathcal{F}$ for each $|\epsilon|=1$.
\end{theorem}

\begin{proof}
If $f=h+\bar{g} \in \mathcal{F}_{H}^{0}$ then for each $|\epsilon|=1$, we have
\[|F_{\epsilon}'(z)-1|\leq|h'(z)-1|+|\epsilon||g'(z)|<1\]
showing that $F_{\epsilon} \in \mathcal{F}$. Conversely, if $F_{\epsilon} \in \mathcal{F}$ for each $|\epsilon|=1$, then the inequality $|h'(z)+\epsilon g'(z)-1|<1$ holds for each $z \in \mathbb{D}$ and $|\epsilon|=1$. With appropriate choice of $\epsilon=\epsilon(z)$, it follows that
\[|h'(z)-1|+|g'(z)|<1\]
so that $f \in \mathcal{F}_{H}^{0}$.
\end{proof}

\begin{theorem}\label{th2.1}
Let $f=h+\bar{g} \in \mathcal{F}_{H}^{0}$ where $h$ and $g$ are given by \eqref{eq1.1} with $b_1=g'(0)=0$. Then
\begin{itemize}
  \item [$(i)$] $|a_{n}|\leq 1/n$ for $n=2,3,\ldots$. Equality holds for some $m\geq 2$ iff $f$ is analytic in $\mathbb{D}$ and
  \[f(z)=z+\frac{z^{m}}{m}.\]
  \item [$(ii)$] $|b_{n}|\leq 1/n$ for $n=2,3,\ldots$. Equality holds for some $m\geq 2$ iff $f$ is of the form
  \[f(z)=z+\frac{\bar{z}^{m}}{m}.\]
  \item [$(iii)$] $|a_{n}|^{2}+|b_{n}|^{2}\leq 1/n^{2}$  for $n=2,3,\ldots$.
  \item [$(iv)$] $|a_n|+|b_n|\leq 1/n$ for $n=2,3,\ldots$.
\end{itemize}
\end{theorem}

\begin{proof}
By Lemma \ref{lem1.1}(c), it is easily seen that
\[\sum_{n=2}^{\infty}n^{2}|a_{n}|^{2} \leq 1 \quad\mbox{and}\quad \sum_{n=2}^{\infty}n^{2}|b_{n}|^{2} \leq 1.\]
This, in particular, shows that $|a_{n}|\leq 1/n$ and $|b_{n}|\leq 1/n$ for $n=2,3,\ldots$. If $|a_{m}|=1/m$ for some $m\geq 2$, then the inequality in Lemma \ref{lem1.1}(c) reduces to
\[m^2 |b_{m}|^{2}+\sum_{\begin{array}{c}
           n=2 \\
           n\neq m
           \end{array}}^{\infty}n^{2}(|a_n|^{2}+|b_n|^{2})=0.\]
This shows that $a_{n}=0$ for all $n \neq m$ and $b_{n}=0$ for all $n$. Thus, $f(z)=z+z^{m}/m$ as desired.

A similar argument shows that if $|b_{m}|=1/m$ for some $m\geq 2$, then $b_{n}=0$ for all $n \neq m$  and $a_{n}=0$ for all $n$. Hence $f(z)=z+\bar{z}^{m}/m$ in this case. This proves $(i)$ and $(ii)$.

The part $(iii)$ of the theorem follows immediately by Lemma \ref{lem1.1}(c). Regarding the proof of $(iv)$, note that $h+\epsilon g \in \mathcal{F}$ for each $|\epsilon|=1$ by Theorem \ref{relation} so that $|a_n+\epsilon b_n|\leq 1/n$ for $n=2,3,\ldots$ (see \cite{macgregor}). In particular, this shows that $|a_n|+|b_n|\leq 1/n$ for $n=2,3,\ldots$.
\end{proof}

\begin{theorem}\label{th2.2}
Every function $f \in \mathcal{F}_{H}^{0}$ satisfies the inequalities
\[|z|-\frac{1}{2}|z|^{2}\leq |f(z)|\leq |z|+\frac{1}{2}|z|^{2},\quad z \in \mathbb{D}.\]
In particular, the range of each function $f \in  \mathcal{F}_{H}^{0}$ contains the disk $|w|<1/2$. Moreover, these results are sharp for the functions $z+z^2/2$ and $z+\bar{z}^{2}/2$.
\end{theorem}

\begin{proof}
Suppose that $f=h+\bar{g} \in \mathcal{F}_{H}^{0}$ and let $F_{\epsilon}=h+\epsilon g$, where $|\epsilon|=1$. Then $F_{\epsilon} \in \mathcal{F}$ by Theorem \ref{relation} so that \cite{macgregor}
\[1-|z|\leq |F_{\epsilon}'(z)|\leq1+|z|,\quad z\in \mathbb{D},\]
or equivalently
\[1-|z|\leq|h'(z)+\epsilon g'(z)|\leq 1+|z|.\]
From this, we may deduce that
\[1-|z|\leq |h'(z)|-|g'(z)|\]
and
\[|h'(z)|+|g'(z)|\leq1+|z|.\]
If $\Gamma$ is the radial segment from $0$ to $z$, then
\begin{align*}
|f(z)|&=\left|\int_{\Gamma} \frac{\partial f}{\partial\zeta}\,d\zeta+\frac{\partial f}{\partial\overline{\zeta}}\,d\overline{\zeta}\right|\leq \int_{\Gamma} (|h'(\zeta)|+|g'(\zeta)|)|d\zeta|\\
      &\leq \int_{0}^{|z|} (1+t)\,dt=|z|+\frac{1}{2}|z|^{2}.
\end{align*}
Next, let $\Gamma$ be the pre-image under $f$ of the radial segment from $0$ to $f(z)$. Then
\begin{align*}
|f(z)|&=\int_{\Gamma}\left| \frac{\partial f}{\partial\zeta}\,d\zeta+\frac{\partial f}{\partial\overline{\zeta}}\,d\overline{\zeta}\right|\geq \int_{\Gamma} (|h'(\zeta)|+|g'(\zeta)|)|d\zeta|\\
      &\geq \int_{0}^{|z|} (1-t)\,dt=|z|-\frac{1}{2}|z|^{2},
\end{align*}
which completes the proof of the theorem.
\end{proof}

\begin{corollary}
The class $\mathcal{F}_{H}^{0}$ is a compact normal family.
\end{corollary}

\begin{proof}
By Theorem \ref{th2.2}, the functions $f \in \mathcal{F}_{H}^{0}$ are uniformly bounded on each compact subset of $\mathbb{D}$. Thus the family $\mathcal{F}_{H}^{0}$ is locally bounded and hence normal, since a family of harmonic functions is normal if it is locally bounded (see \cite[Theorem 2.3]{cluniesheilsmall}).

To prove the compactness of $\mathcal{F}_{H}^{0}$, suppose that $f_n=h_n+\overline{g_n} \in \mathcal{F}_{H}^{0}$ and that $f_n\rightarrow f$ uniformly on compact subsets of $\mathbb{D}$. Then $f$ is harmonic and so $f=h+\bar{g}$. It is easy to see that $h_n\rightarrow h$ and $g_n\rightarrow g$ locally uniformly so that the condition
\[|h'_n(z)-1|<1-|g'_n(z)|,\quad z \in \mathbb{D}, \quad n=1,2,\ldots\]
implies that $|h'(z)-1|<1-|g'(z)|$ for $z \in \mathbb{D}$ as $n\rightarrow\infty$. Thus the limit function $f$ again belongs to $\mathcal{F}_{H}^{0}$. This concludes the proof.
\end{proof}

\begin{theorem}\label{th2.4}
If $f \in \mathcal{F}_{H}^{0}$ then the Jacobian of $f$ satisfy
\[J_{f}(z)\leq(1+|z|)^{2},\quad z \in \mathbb{D}.\]
Equality occurs only if $f$ is analytic and is a rotation of the function $f(z)=z+z^2/2$.
\end{theorem}

\begin{proof}
Let $f=h+\bar{g}$. Then $h \in \mathcal{F}$ and
\begin{equation}\label{eq2.1}
J_{f}(z)=|h'(z)|^{2}-|g'(z)|^{2}\leq |h'(z)|^2\leq(1+|z|)^{2},\quad z \in \mathbb{D}.
\end{equation}
For the proof of equality, suppose that there is a $z_{0}\in \mathbb{D}$ for which $J_{f}(z_{0})=(1+|z_{0}|)^{2}$. Then equality occurs throughout in \eqref{eq2.1} so that $|g'(z_{0})|=0$ and $|h'(z_{0})|=1+|z_{0}|$. By applying Schwarz Lemma to the function $h'(z)-1$, we have
\[|h'(z_{0})|\leq 1+|h'(z_{0})-1|\leq 1+|z_{0}|.\]
Thus, we deduce that $|h'(z_{0})-1|=|z_{0}|$. Again, Schwarz Lemma shows that $h'(z)-1=e^{i\theta}z$ for some $\theta \in \mathbb{R}$ and hence $h(z)=z+e^{i\theta}z^{2}/2$. As $|a_{2}|=1/2$, $g\equiv 0$ in $\mathbb{D}$ by Theorem \ref{th2.1}$(i)$ so that $f(z)=z+e^{i\theta}z^{2}/2$.
\end{proof}

It is known that (see \cite[Chapter 6]{duren}) if $f\in \mathcal{S}_{H}^{0}$ then the area of $f(\mathbb{D})$ is greater than or equal to $\pi/2$. Furthermore, the minimum area is attained for the function $f(z)=z+\bar{z}^{2}/2$ and its rotations. This, in particular, shows that the class $\mathcal{F}_{H}^{0}$ contains area-minimizing functions. The next theorem demonstrates that $\mathcal{F}_{H}^{0}$ does contain area-maximizing functions also.

\begin{theorem}\label{th2.5}
The area of the image of each function $f$ in $\mathcal{F}_{H}^{0}$ is less than or equal to $3\pi/2$ and this is a maximum attained only by the analytic functions $f(z)=z+z^2/2$ and its rotations.
\end{theorem}

\begin{proof}
Suppose that $f=h+\bar{g} \in \mathcal{F}_{H}^{0}$, where $h$ and $g$ are given by \eqref{eq1.1}. Then the area of the image $f(\mathbb{D})$ is
\begin{equation}\label{eq2.2}
\begin{split}
A&=\int \!\!\! \int_{\mathbb{D}}J_{f}(z)\,dx\,dy=\int \!\!\! \int_{\mathbb{D}}(|h'(z)|^{2}-|g'(z)|^{2})\,dx\,dy\\
 &=\pi\left[1+\sum_{n=2}^{\infty}n(|a_{n}|^{2}-|b_{n}|^{2})\right]\leq\pi\left[1+\sum_{n=2}^{\infty}n(|a_{n}|^{2}+|b_{n}|^{2})\right]\\
 &\leq \pi\left[1+\frac{1}{2}\sum_{n=2}^{\infty}n^{2}(|a_{n}|^{2}+|b_{n}|^{2})\right]\leq \frac{3\pi}{2},
\end{split}
\end{equation}
using Lemma \ref{lem1.1}(c). If $A=3\pi/2$ for a function $f\in \mathcal{F}_{H}^{0}$ then equality occurs throughout in \eqref{eq2.2}. This shows that $b_{n}=0$ for all $n\geq 2$, $a_{n}=0$ for all $n \geq 3$ and $|a_{2}|=1/2$. Hence $f$ is a rotation of the function $z+z^2/2$.
\end{proof}

\section{Boundary Behavior, convexity and starlikeness}\label{sec3}
In this section, it has been shown that the boundary of $f(\mathbb{D})$ is a rectifiable Jordan curve for each $f \in \mathcal{F}_{H}^{0}$. We will also determine the radius of convexity and starlikeness of the class $\mathcal{F}_{H}^{0}$.

\begin{theorem}\label{th3.1}
Each function in $\mathcal{F}_{H}^{0}$ maps $\mathbb{D}$ onto a domain bounded by a rectifiable Jordan curve.
\end{theorem}

\begin{proof}
Each $f=h+\bar{g} \in \mathcal{F}_{H}^{0}$ is uniformly continuous in $\mathbb{D}$ and hence can be extended continuously onto $|z|=1$. To see this, let $z_1$ and $z_2$ be distinct points in $\mathbb{D}$. If ${[z_1.z_2]}$ is the line segment from $z_1$ to $z_2$, then
\begin{align*}
|f(z_1)-f(z_2)|&=\left|\int_{[z_1.z_2]} f_z(z)dz+f_{\bar{z}}(z)d\bar{z}\right|=\left|\int_{[z_1.z_2]} h'(z) dz+g'(z)d\bar{z}\right|\\
               &\leq\int_{[z_1.z_2]}(|h'(z)|+|g'(z)|)|dz|\leq 2|z_1-z_2|,
\end{align*}
using the proof of Theorem \ref{th2.2}. Let $C$ be defined by $w=f(e^{i \theta})$, $0\leq \theta \leq 2\pi$.

If $0=\theta_0<\theta_1<\theta_2\cdots <\theta_n=2\pi$ is a partition of $[0,2\pi]$, then
\[\sum_{k=1}^{n}|f(e^{i\theta_{k}})-f(e^{i\theta_{k-1}})|\leq 2\sum_{k=1}^{n}|e^{i\theta_{k}}-e^{i\theta_{k-1}}|<4\pi\]
showing that $C$ is rectifiable.

Next, it remains to show that $f$ is univalent in $\partial \mathbb{D}=\{z\in \mathbb{C}:|z|=1\}$. By Theorem \ref{relation}, the analytic functions $F_\lambda=h+\lambda g$ belongs to the class $\mathcal{F}$ for all $|\lambda|=1$. In particular, each $F_\lambda$ is univalent in $\partial\mathbb{D}$ by \cite[Theorem 3]{macgregor}. Suppose that $z_1$, $z_2 \in \mathbb{D}$ be such that $f(z_1)=f(z_2)$. Then
\[h(z_1)-h(z_2)=\overline{g(z_2)-g(z_1)}.\]
If $h(z_1)=h(z_2)$ then $g(z_1)=g(z_2)$ so that $z_1=z_2$ by the univalence of $F_1$. Therefore, assume that $h(z_1)\neq h(z_2)$. Denote by
\[\theta=\arg \{h(z_1)-h(z_2)\}\in [0,2\pi).\]
Then, it follows that
\[e^{-i \theta}(h(z_1)-h(z_2))=\overline{e^{i \theta}(g(z_2)-g(z_1))}\]
is a positive real number. Thus on taking conjugates, we have
\[h(z_1)-h(z_2)=e^{2i \theta}(g(z_2)-g(z_1)),\]
which imply that $F_{\mu}(z_1)=F_{\mu}(z_2)$, $\mu=e^{2i \theta}$. Hence $z_1=z_2$ and this proves that $f$ is univalent in $\partial\mathbb{D}$.
\end{proof}

The next theorem determines the radius of convexity of the class $\mathcal{F}_{H}^{0}$.
\begin{theorem}\label{th3.2}
The radius of convexity of the class $\mathcal{F}_{H}^{0}$ is $1/2$. Moreover, the bound $1/2$ is sharp.
\end{theorem}

\begin{proof}
Let $f=h+\bar{g} \in \mathcal{F}_{H}^{0}$. Then the analytic functions $h-e^{i\phi} g$, $0\leq\phi<2\pi$ belongs to the class $\mathcal{F}$. Consequently the functions $h-e^{i\phi} g$ are convex in $|z|<1/2$ by \cite[Theorem 5]{macgregor}. In view of \cite[Theorem 5.7]{cluniesheilsmall}, it follows that $f$ is convex in $|z|<1/2$. The functions $z+z^{2}/2$ and $z+\bar{z}^{2}/2$ shows that the bound $1/2$ is best possible.
\end{proof}

Theorem \ref{th3.2} shows that the classes $\mathcal{F}$ and $\mathcal{F}_{H}^{0}$ have the same radius of convexity. A similar statement holds regarding the radius of starlikeness as seen by the following theorem.

\begin{theorem}\label{th3.3}
The classes $\mathcal{F}$ and $\mathcal{F}_{H}^{0}$ have the same radius of starlikeness.
\end{theorem}

\begin{proof}
Since $\mathcal{F}\subset \mathcal{F}_{H}^{0}$ therefore it suffices to show that if $r_0$ is the radius of starlikeness of $\mathcal{F}$, then $f$ is starlike in $|z|<r_{0}$ for each $f \in \mathcal{F}_{H}^{0}$. To see this, suppose that $f=h+\bar{g} \in \mathcal{F}_{H}^{0}$. Then the analytic functions $F_{\lambda}=f+\lambda g$ belongs to $\mathcal{F}$ for each $\lambda \in \partial \mathbb{D}$. Consequently, each $F_{\lambda}$ is starlike in $|z|<r_0$. By \cite[Theorem 3, p. 10]{stable}, it follows that $f$ is starlike in $|z|<r_0$.
\end{proof}

Note that the exact radius of starlikeness for the class $\mathcal{F}$ is still unknown. In \cite{macgregor}, MacGregor showed that $r_0\geq 2/\sqrt{5}$, $r_0$ being the radius of starlikeness of $\mathcal{F}$. In \cite{singh}, it has been shown that $r_{0}>0.974$ and by \cite{nunokawa2}, $r_0<0.9982$.

In \cite{singh}, Singh considered the subclass $\mathcal{F}(\lambda)$ $(0<\lambda\leq 1)$ of $\mathcal{F}$ consisting of functions which satisfy $|f'(z)-1|<\lambda$ for $z \in \mathbb{D}$. He showed that $\mathcal{F}(\lambda)\subset \mathcal{S}^{*}$ for $0<\lambda\leq 2/\sqrt{5}$. In \cite{fournier}, it has been shown that the bound $2/\sqrt{5}$ is best possible. Analogously, we may define a subclass $\mathcal{F}_{H}(\lambda)$ $(0<\lambda\leq 1)$ of $\mathcal{F}_{H}$ defined by the condition
\[|f_{z}(z)-1|<\lambda-|f_{\bar{z}}(z)|,\quad z\in \mathbb{D}.\]
The next theorem characterizes the functions in the class $\mathcal{F}_{H}^{0}(\lambda)=\{f \in \mathcal{F}_{H}(\lambda):f_{\bar{z}}(0)=0\}$. Its proof is easy to obtain and so its details are omitted.

\begin{theorem}\label{th3.4}
Suppose that $f=h+\bar{g} \in \mathcal{F}_{H}^{0}(\lambda)$ $(0<\lambda\leq 1)$ where $h$ and $g$ are given by \eqref{eq1.1} with $b_{1}=g'(0)=0$. Then
\begin{itemize}
\item[$(i)$] The coefficients of $h$ and $g$ satisfies
\[\sum_{n=2}^{\infty}n^{2}(|a_n|^{2}+|b_n|^{2})\leq \lambda^{2}.\]
In particular, $|a_{n}|\leq \lambda/n$ and $|b_{n}|\leq \lambda/n$ for $n=2,3,\ldots$.
\item[$(ii)$] A harmonic function $f=h+\bar{g} \in \mathcal{F}_{H}^{0}(\lambda)$ if and only if the analytic functions $F_{\epsilon}=h+\epsilon g$ belongs to $\mathcal{F}(\lambda)$ for each $|\epsilon|=1$.
\item[$(iii)$] (Growth estimate)
\[|z|-\frac{1}{2}\lambda|z|^{2}\leq |f(z)|\leq |z|+\frac{1}{2}\lambda|z|^{2},\quad z \in \mathbb{D}.\]
In particular, each function in $\mathcal{F}_{H}^{0}(\lambda)$ assumes every complex number in the disk $|w|<1-\lambda/2$. The function $f(z)=z+\lambda \bar{z}^{2}/2$ shows that these results are sharp.
\item[$(iv)$] The area $A$ of the image domain under each function in $\mathcal{F}_{H}^{0}(\lambda)$ satisfies $A\leq \pi (1+\lambda^{2}/2)$ and  $A=\pi (1+\lambda^{2}/2)$ only for the analytic function $f(z)=z+\lambda z^{2}/2$ and its rotations.
\end{itemize}
\end{theorem}

The next theorem shows that the functions in the class $\mathcal{F}_{H}^{0}(\lambda)$ are starlike for $0<\lambda\leq 2/\sqrt{5}$.
\begin{theorem}\label{th3.5}
$\mathcal{F}_{H}^{0}(2/\sqrt{5})\subset \mathcal{S}_{H}^{*0}$ and the bound $2/\sqrt{5}$ is best possible.
\end{theorem}

\begin{proof}
Suppose that $f=h+\bar{g} \in \mathcal{F}_{H}^{0}(2/\sqrt{5})$. Then the analytic functions $h+\epsilon g$ belongs to $\mathcal{F}(2/\sqrt{5})$ by Theorem \ref{th3.4}(ii) and hence are starlike for each $|\epsilon|=1$ by \cite{singh}. By \cite[Theorem 3, p. 10]{stable}, it follows that $f \in \mathcal{S}_{H}^{*0}$.
\end{proof}

Similar to Lemma \ref{lem1.2}, the next theorem provides a sufficient condition for a harmonic function to be in $\mathcal{F}_{H}^{0}(\lambda)$.
\begin{theorem}\label{th3.6}
Let $f=h+\bar{g} \in \mathcal{H}$ where $h$ and $g$ are given by \eqref{eq1.1} with $b_1=g'(0)=0$. Suppose that $\lambda \in (0,1]$. If
\begin{equation}\label{eq3.1}
\sum_{n=2}^{\infty}n(|a_n|+|b_n|)\leq \lambda
\end{equation}
then $f \in \mathcal{F}_{H}^{0}(\lambda)$ and is starlike of order $2(1-\lambda)/(2+\lambda)$. The result is sharp.
\end{theorem}

\begin{proof}
Using the hypothesis, we have
\[|h'(z)-1|\leq \sum_{n=2}^{\infty}n|a_n||z|^{n-1}<\sum_{n=2}^{\infty}n|a_n|\leq \lambda-\sum_{n=2}^{\infty}n|b_n|<\lambda-|g'(z)|,\]
which imply that $f \in \mathcal{F}_{H}^{0}(\lambda)$. Regarding the order of starlikeness of $f$, if we set $\alpha=2(1-\lambda)/(2+\lambda)$ then $\alpha \in [0,1)$ and
\begin{align*}
     \sum_{n=2}^{\infty}\left(\frac{n-\alpha}{1-\alpha}|a_{n}|+\frac{n+\alpha}{1-\alpha}|b_{n}|\right)&\leq \sum_{n=2}^{\infty}\frac{n+\alpha}{1-\alpha}(|a_n|+|b_n|)\\
                                  &\leq\frac{1}{1-\alpha}\sum_{n=2}^{\infty} n(a_n|+|b_n|)+ \frac{1}{2}\frac{\alpha}{1-\alpha}\sum_{n=2}^{\infty}n(|a_n|+|b_n|)\\
                                  &\leq \frac{2+\alpha}{2(1-\alpha)}\lambda=1.
\end{align*}
By \cite{jahangiristarlike}, it follows that $f$ is starlike of order $2(1-\lambda)/(2+\lambda)$. The harmonic mapping $f_{0}(z)=z+\lambda\bar{z}^2/2$ shows that the coefficient bound given by \eqref{eq3.1} is sharp. Further, for $z=r e^{i\theta}$, we have
\[\frac{\partial}{\partial \theta}\arg f_{0}(r e^{i \theta})=\RE \frac{2(z-\lambda\bar{z}^{2})}{2z+\lambda\bar{z}^{2}}\geq \frac{2(1-\lambda|z|)}{2+\lambda|z|}>\frac{2(1-\lambda)}{2+\lambda}\]
which shows that the bound for the order of starlikeness in the theorem is also sharp.
\end{proof}

Making use of the following corollary, the convolution properties of functions in the class $\mathcal{F}_{H}^{0}(\lambda)$ $(0<\lambda\leq 1)$ are discussed in the next section.
\begin{corollary}\label{cor3.7}
Let $f=h+\bar{g} \in \mathcal{H}$ where $h$ and $g$ are given by \eqref{eq1.1} with $b_1=g'(0)=0$. Suppose that $\lambda \in (0,1]$. If
\[\sum_{n=2}^{\infty}n^2(|a_n|+|b_n|)\leq \lambda\]
then $f \in \mathcal{F}^{0}_{H}(\lambda/2)$ and is starlike of order $2(2-\lambda)/(4+\lambda)$. Moreover, $f$ is convex of order $2(1-\lambda)/(2+\lambda)$. These results are sharp for the function $f(z)=z+\lambda\bar{z}^{2}/4$.
\end{corollary}

\begin{proof}
Observe that
\[\sum_{n=2}^{\infty}n(|a_n|+|b_n|)\leq \frac{1}{2}\sum_{n=2}^{\infty}n^2(|a_n|+|b_n|)\leq\frac{\lambda}{2}.\]
By Theorem \ref{th3.6}, $f \in \mathcal{F}^{0}_{H}(\lambda/2)$ and is starlike of order $2(2-\lambda)/(4+\lambda)$. For the order of convexity, setting $\beta=2(1-\lambda)/(2+\lambda)$, we see that
\[\sum_{n=2}^{\infty}\left(\frac{n(n-\beta)}{1-\beta}|a_{n}|+\frac{n(n+\beta)}{1-\beta}|b_{n}|\right)\leq \frac{(2+\beta)}{2(1-\beta)}\lambda=1\]
By \cite{jahangiriconvex}, $f$ is convex of order $2(1-\lambda)/(2+\lambda)$.
\end{proof}

\section{Convolution and Convex Combinations}\label{sec4}
In this section, it has been shown that the class $\mathcal{F}_{H}^{0}$ is invariant under convolution and convex combinations of its members.

\begin{theorem}\label{th4.1}
Let $f$, $F \in \mathcal{F}_{H}^{0}(\lambda)$ $(0<\lambda\leq 1)$. Then
\begin{itemize}
  \item [$(a)$] $f*F \in \mathcal{F}^{0}_{H}(\lambda^2/2)$;
  \item [$(b)$] $f*F$ is starlike of order $2(2-\lambda^{2})/(4+\lambda^{2})$; and
  \item [$(c)$] $f*F$ is convex of order $2(1-\lambda^{2})/(2+\lambda^{2})$.
\end{itemize}
Moreover, all these results are sharp.
\end{theorem}

\begin{proof}
Write
\begin{equation}\label{eq4.1}
f(z)=z+\sum_{n=2}^{\infty}a_{n}z^{n}+\overline{\sum_{n=2}^{\infty}b_{n}z^{n}}\quad \mbox{ and }\quad F(z)=z+\sum_{n=2}^{\infty}A_{n}z^{n}+\overline{\sum_{n=2}^{\infty}B_{n}z^{n}}.
\end{equation}
Then the convolution of $f$ and $F$ is given by
\[(f*F)(z)=z+\sum_{n=2}^{\infty}a_{n}A_{n}z^{n}+\overline{\sum_{n=2}^{\infty}b_{n}B_{n}z^{n}}.\]
Using the fact that the geometric mean is less than or equal to the arithmetic mean and applying Theorem \ref{th3.4}(i), we have
\begin{align*}
\sum_{n=2}^{\infty}n^{2}(|a_n A_n|+|b_n B_n|)&\leq \sum_{n=2}^{\infty}n^{2}\left(\frac{|a_n|^2+|A_n|^2}{2}+\frac{|b_n|^2+|B_n|^2}{2}\right)\\
                                         &=\frac{1}{2}\sum_{n=2}^{\infty}n^{2}(|a_n|^{2}+|b_n|^{2})+\frac{1}{2}\sum_{n=2}^{\infty}n^{2}(|A_n|^{2}+|B_n|^{2})\leq \lambda^{2}.
\end{align*}
By Corollary \ref{cor3.7}, the result follows immediately. For sharpness, consider the function $f(z)=F(z)=z+\lambda\bar{z}^{2}/2$ whose convolution is given by $(f*F)(z)=z+\lambda^{2} \bar{z}^{2}/4$.
\end{proof}

Taking $\lambda=1$ in Theorem \ref{th4.1}, we have
\begin{corollary}\label{cor4.2}
If $f$ and $F$ belong to $\mathcal{F}_{H}^{0}$ then so does the convolution function $f*F$. Moreover,
\begin{enumerate}
  \item [$(i)$] $f*F$ is starlike of order $2/5$; and
  \item [$(ii)$] $f*F$ is convex in $\mathbb{D}$.
\end{enumerate}
All these results are sharp by considering the function $f(z)=F(z)=z+\bar{z}^{2}/2$.
\end{corollary}

In \cite{cluniesheilsmall}, Clunie and Sheil-Small showed that if $\varphi \in \mathcal{K}$ and $f \in \mathcal{K}_{H}$ then the functions $(\alpha \overline{\varphi}+\varphi)*f \in \mathcal{C}_{H}$ ($|\alpha|\leq 1$). The result is even true if $\mathcal{K}_{H}$ is replaced by $\mathcal{F}_{H}^{0}$ with a stronger conclusion. This is seen by the following theorem which makes use of the result due to Ruscheweyh \cite[Theorem 2.4]{ruscheweyh} which states that if $f \in \mathcal{S}^{*}$ and $\phi \in \mathcal{K}$ then
\[\frac{\varphi *f F}{\varphi*f}(\mathbb{D}) \subset \overline{\mbox{co}}(F(\mathbb{D}))\]
for any function $F$ analytic in $\mathbb{D}$, where $\overline{co}$$(F(\mathbb{D}))$ denotes the closed convex hull of $F(\mathbb{D})$.

\begin{theorem}\label{th4.3}
Let $\varphi \in \mathcal{K}$ and $f \in \mathcal{F}_{H}^{0}(\lambda)$ $(0<\lambda\leq 1)$. Then the functions $(\alpha \overline{\varphi}+\varphi)*f \in \mathcal{F}^{0}_{H}(\lambda)$ for $|\alpha|\leq 1$.
\end{theorem}

\begin{proof}
Writing $f=h+\bar{g}$ we have
\[(\alpha \overline{\varphi}+\varphi)*f=\varphi*h+\overline{\bar{\alpha}(\varphi*g)}=H+\overline{G},\]
where $H=\varphi*h$ and $G=\bar{\alpha}(\varphi*g)$ are analytic in $\mathbb{D}$. Setting $F=H+\epsilon G=\varphi*(h+\epsilon \bar{\alpha}g)$ where $|\epsilon|=1$, we note that $F \in \mathcal{F}(\lambda)$. To see this, consider
\[F'(z)=\frac{(\varphi*(h+\epsilon \bar{\alpha}g)*k)(z)}{z}=\frac{(\varphi* z(h+\epsilon \bar{\alpha}g)')(z)}{(\varphi*z)(z)}\in \overline{\mbox{co}}((h+\epsilon \bar{\alpha}g)'(\mathbb{D})), \]
where $k(z)=z/(1-z)^{2}$ is the Koebe function. This shows that $|F'(z)-1|<\lambda$ for all $z \in \mathbb{D}$ since $h+\epsilon \bar{\alpha}g \in \mathcal{F}(\lambda)$. Thus $F=H+\epsilon G$ belongs to $\mathcal{F}(\lambda)$ for each $|\epsilon|=1$. By using Theorem \ref{th3.4}(ii), it follows that $H+\bar{G} \in \mathcal{F}^{0}_{H}(\lambda)$, as desired.
\end{proof}

Goodloe \cite{goodloe} considered the Hadamard product $\tilde{*}$ of a harmonic function with an analytic function defined as follows:
\[f \tilde{*} \varphi=\varphi \tilde{*} f=h*\varphi+\overline{g*\varphi},\]
where $f=h+\bar{g}$ is harmonic and $\varphi$ is analytic in $\mathbb{D}$. By Theorem \ref{th4.3}, it follows that $\varphi \tilde{*} f \in \mathcal{F}^{0}_{H}$ if $\varphi \in \mathcal{K}$ and $f \in \mathcal{F}_{H}^{0}$. Further properties of the product $\tilde{*}$ are investigated in the following corollaries.

\begin{corollary}\label{cor4.4}
Let $\varphi \in \mathcal{K}$ and $f \in \mathcal{F}_{H}^{0}$. Then
\begin{enumerate}
  \item [$(i)$] $\varphi \tilde{*} f$ is starlike in $|z|<r_{0}$, $r_{0}$ being the radius of starlikeness of $\mathcal{F}$; and
  \item [$(ii)$] $\varphi \tilde{*} f$ is convex in $|z|<1/2$.
\end{enumerate}
\end{corollary}

\begin{proof}
The proof follows by using Theorems \ref{th3.2} and \ref{th3.3} since $\varphi \tilde{*} f \in \mathcal{F}^{0}_{H}$.
\end{proof}

Considering the functions $\varphi(z)=z/(1-z)\in \mathcal{K}$ and $f(z)=z+\bar{z}^{2}/2\in \mathcal{F}_{H}^{0}$, observe that $\varphi \tilde{*} f=f$ which shows that the bound $1/2$ in Corollary \ref{cor4.4}(ii) is best possible.

\begin{corollary}\label{cor4.5}
If $\varphi \in \mathcal{K}$ and $f \in \mathcal{F}_{H}^{0}(2/\sqrt{5})$ then $\varphi \tilde{*} f \in \mathcal{S}_{H}^{*0}$.
\end{corollary}

\begin{proof}
Note that $\varphi \tilde{*} f \in \mathcal{F}^{0}_{H}(2/\sqrt{5}) \subset \mathcal{S}_{H}^{*0}$ by applying Theorems \ref{th3.5} and \ref{th4.3}.
\end{proof}

\begin{corollary}\label{cor4.6}
Let $\varphi \in \mathcal{F}$ and $f \in \mathcal{F}_{H}^{0}$. Then $\varphi \tilde{*} f \in \mathcal{F}_{H}^{0} \cap \mathcal{K}_{H}^{0}$.
\end{corollary}

\begin{proof}
The proof of Theorem \ref{th4.3} shows that the functions $H+\epsilon G=\varphi*(h+\epsilon g)$ $(\alpha=1, |\epsilon|=1)$ belongs to $\mathcal{F}$ and are convex in $\mathbb{D}$ by Corollary \ref{cor4.2} since both the analytic functions $h+\epsilon g$ and $\varphi$ are in $\mathcal{F}$. By \cite[Theorem 5.7]{cluniesheilsmall}, $\varphi \tilde{*} f \in \mathcal{K}_{H}^{0}$. Moreover, $f \in \mathcal{F}_{H}^{0}$ by Theorem \ref{relation}.
\end{proof}

Given $f$, $F \in \mathcal{F}_{H}^{0}$, their integral convolution is defined as
\[(f \diamond F)(z)=z+\sum_{n=2}^{\infty}\frac{a_{n}A_{n}}{n}z^{n}+\overline{\sum_{n=2}^{\infty}\frac{b_{n}B_{n}}{n}z^{n}}, \quad z \in \mathbb{D}\]
The proof of the next theorem follows from an easy modification of the proof of Theorem \ref{th4.1}.

\begin{theorem}\label{th4.7}
Let $f$, $F \in \mathcal{F}_{H}^{0}(\lambda)$ $(0<\lambda\leq 1)$. Then
\begin{itemize}
  \item [$(a)$] $f\diamond F \in \mathcal{F}^{0}_{H}(\lambda^2/4)$;
  \item [$(b)$] $f\diamond F$ is starlike of order $2(4-\lambda^{2})/(8+\lambda^{2})$; and
  \item [$(c)$] $f\diamond F$ is convex of order $2(2-\lambda^{2})/(4+\lambda^{2})$.
\end{itemize}
Moreover, all these results are sharp.
\end{theorem}

If $\lambda=1$ in Theorem \ref{th4.7}, then we obtain the integral convolution properties for the functions in the class $\mathcal{F}_{H}^{0}$.
\begin{corollary}\label{cor4.8}
If $f$, $F \in \mathcal{F}_{H}^{0}$ then so does $f \diamond F$. Moreover
\begin{enumerate}
  \item [$(i)$] $f\diamond F$ is starlike of order $2/3$; and
  \item [$(ii)$] $f\diamond F$ is convex of order $2/5$.
\end{enumerate}
These results are sharp by considering the function $f(z)=F(z)=z+\bar{z}^{2}/2$.
\end{corollary}

Now, we determine the convex combination properties of the members of class $\mathcal{F}_{H}^{0}$.
\begin{theorem}\label{th4.9}
The class $\mathcal{F}_{H}^{0}$ is closed under convex combinations.
\end{theorem}

\begin{proof}
For $n=1,2,\ldots$, suppose that $f_n \in \mathcal{F}_{H}^{0}$ where $f_n=h_n+\overline{g_n}$. For $\sum_{n=1}^{\infty}t_n=1$, $0\leq t_n \leq 1$, the convex combination of $f_n$'s may be written as
\[f(z)=\sum_{n=1}^{\infty}t_n f_n(z)=h(z)+\overline{g(z)},\]
where
\[h(z)=\sum_{n=1}^{\infty}t_n h_n(z)\quad \mbox{and} \quad g(z)=\sum_{n=1}^{\infty}t_n g_n(z).\]
Then $h$ and $g$ are analytic in $\mathbb{D}$ with $h(0)=g(0)=h'(0)-1=g'(0)=0$ and
\[|h'(z)-1|=\left|\sum_{n=1}^{\infty}t_n (h_{n}'(z)-1)\right|\leq\sum_{n=1}^{\infty}t_{n}|h_{n}'(z)-1|\leq \sum_{n=1}^{\infty}t_n(1-|g_{n}'(z)|)\leq1-|g'(z)|,\]
showing that $f \in \mathcal{F}_{H}^{0}$.
\end{proof}

Theorem \ref{th4.9} immediately gives

\begin{corollary}\label{cor3.10}
The class $\mathcal{F}_{H}^{0}$ is convex.
\end{corollary}

We end this section with a simple observation regarding the neighborhoods of harmonic mappings. Following \cite{nbd}, if we define the $\delta$-neighborhood $(\delta \geq 0)$ of a function $f=h+\bar{g} \in \mathcal{H}$  by
\begin{align*}
N_{\delta}(f)&=\left\{F\in \mathcal{H}:F(z)=z+\sum_{n=2}^{\infty}A_{n}z^{n}+\overline{\sum_{n=1}^{\infty} B_{n}z^{n}}\right.\\
                                     &\quad \quad \mbox{and }\left.\sum_{n=2}^{\infty}n (|a_n-A_n|+|b_n-B_n|)+|b_1-B_1|\leq \delta\right\},
\end{align*}
 $h$ and $g$ being given by \eqref{eq1.1}, then for the identity function $e(z)=z$, we immediately have $I_1(e)\subset \mathcal{F}_{H}\cap \mathcal{S}^{*}_{H}$ by applying Lemma \ref{lem1.2}.

\end{document}